\documentclass{amsart}
\usepackage{amsmath,amssymb}

\newtheorem{theorem}{Theorem}[section]

\newtheorem{corollary}[theorem]{Corollary}

\newtheorem{proposition}[theorem]{Proposition}
\newtheorem{example}[theorem]{Example}

\numberwithin{equation}{section}

\newcommand{\N}{\mathbb{N}}
\newcommand{\Z}{\mathbb{Z}}
\newcommand{\C}{\mathbb{C}}
\newcommand{\T}{\mathbb{T}}
\newcommand{\D}{\mathbb{D}}

\newcommand{\supp}{\textup{supp}}

\begin{document}

\title[An areal analog of Mahler's measure]
{An areal analog of Mahler's measure}%
\author{Igor E. Pritsker}%

\thanks{Research was partially supported by the National Security
Agency under grant H98230-06-1-0055, and by the Alexander von
Humboldt Foundation.}

\address{Department of Mathematics, Oklahoma State University,
Stillwater, OK 74078, U.S.A.}%
\email{igor@math.okstate.edu}




\begin{abstract}
We consider a version of height on polynomial spaces defined by the
integral over the normalized area measure on the unit disk. This
natural analog of Mahler's measure arises in connection with
extremal problems for Bergman spaces. It inherits many nice
properties such as the multiplicative one. However, this height is a
lower bound for Mahler's measure, and it can be substantially lower.
We discuss some similarities and differences between the two.
\end{abstract}

\maketitle


\section{Definition and main properties}

Let $\C_n[z]$ and $\Z_n[z]$ be the sets of all polynomials of degree
at most $n$ with complex and integer coefficients respectively.
Mahler's measure of a polynomial $P_n\in\C_n[z]$ is defined by
\[
M(P_n) := \|P_n\|_{H^0} = \exp\left(\frac{1}{2\pi}\int_0^{2\pi}
\log|P_n(e^{i\theta})|\,d\theta\right).
\]
It is also known as the $H^0$ Hardy space norm or the contour
geometric mean. An application of Jensen's inequality immediately
gives that
\[
M(P_n) = |a_n| \prod_{|z_j|>1} |z_j|
\]
for $P_n(z)=a_n \prod_{j=1}^n (z-z_j)\in\C_n[z].$ This height on the
space of polynomials is extensively used in number theory. Recall
that a cyclotomic (circle dividing) polynomial is defined as an
irreducible factor of $z^n-1,\ n\in\N.$ Clearly, if $Q_n$ is
cyclotomic, then $M(Q_n)=1.$ A well known and difficult open problem
related to Mahler's measure is the Lehmer conjecture on the lower
bound for the measure of irreducible non-cyclotomic polynomials from
$\Z_n[z].$ Lehmer \cite{Leh} carried out extensive computations of
the values of $M(P_n)$, $P_n\in\Z_n[z]$, but found no non-cyclotomic
polynomial with Mahler's measure smaller than that of the polynomial
$L(z):=z^{10}+z^9-z^7-z^6-z^5-z^4-z^3+z+1,$ which he conjectured to
be always true. Eight zeros of $L$ lie on the unit circle, one
inside and one outside. The latter zero $\zeta$ is believed to be
the smallest Salem number, with the value
$M(L)=\zeta=1.1762808\ldots$, which is the smallest Mahler's measure
according to the Lehmer conjecture. More history of this conjecture
may be found in \cite{EW}, \cite{Boy}, \cite{GH} and \cite{Bor}.

A natural counterpart of Mahler's measure is obtained by replacing
the normalized arclength measure on the unit circumference $\T$ by
the normalized area measure on the open unit disk $\D.$ Namely, we
define the $A^0$ Bergman space norm by
\[
\|P_n\|_0 := \exp\left(\frac{1}{\pi} \iint_{\D} \log|P_n(z)|
\,dA\right).
\]
This also gives a multiplicative height of the polynomial $P_n.$
Furthermore, it has the same relation to Bergman spaces as Mahler's
measure to Hardy spaces:
\[
\|P_n\|_0 = \lim_{p\to0+} \|P_n\|_p,
\]
see \cite{HLP}, where
\[
\|P_n\|_p := \left(\frac{1}{\pi} \iint_{\D} |P_n(z)|^p
\,dA\right)^{1/p}, \quad 0<p<\infty.
\]
In addition, it arises in the following version of the extremal
problem considered by Szeg\H{o} \cite{Sz20} for the Hardy space
$H^2$:
\[
\inf_{Q(0)=0} \frac{1}{\pi} \iint_{\D} |1-Q(z)|^p\, |P_n(z)|\, dA(z)
= \|P_n\|_0,\quad 0<p<\infty,
\]
where $Q$ is any polynomial vanishing at $0$, see \cite[p. 136]{Ga}.

Using the fact that the integral means of $\log|P_n(re^{it})|$ over
$|z|=r$ are increasing with $r$ \cite{DS}, we immediately obtain
that
\begin{align} \label{1.1}
\|P_n\|_0 \le M(P_n).
\end{align}
Also, if $P_n(z)=\sum_{k=0}^n a_k z^k$ then
\begin{align} \label{1.2}
\|P_n\|_0 \ge |a_0|,
\end{align}
which follows from the area mean value inequality for the
subharmonic function $\log|P_n|$ (cf. \cite{DS}). Hence
\begin{align} \label{1.3}
\|P_n\|_0 \ge 1 \quad\mbox{for all } P_n\in\Z_n[z],\ P_n(0)\neq 0.
\end{align}
In fact, there is a direct relation between Mahler's measure and its
areal analog, given below.

\begin{theorem} \label{thm1.1}
Let $P_n(z) = a_n \prod_{j=1}^n (z-z_j) = \sum_{k=0}^n a_k z^k
\in\C_n[z].$ If $P_n$ has no roots in $\D,$ then $\|P_n\|_0 = M(P_n)
=|a_0|.$ Otherwise,
\begin{align} \label{1.4}
\|P_n\|_0 = M(P_n)\, \exp\left(\frac{1}{2} \sum_{|z_j| < 1}
(|z_j|^2-1) \right).
\end{align}
\end{theorem}

This shows that the value of $\|P_n\|_0$ is influenced by the zeros
inside the unit disk more than that of $M(P_n).$ We immediately
obtain the following comparison result from Theorem \ref{thm1.1}.

\begin{corollary} \label{cor1.2}
For any $P_n\in\C_n[z],$ we have
\begin{align} \label{1.5}
e^{-n/2} M(P_n) \le \|P_n\|_0 \le M(P_n).
\end{align}
Equality holds in the lower estimate if and only if $P_n(z)= a_n
z^n.$ The upper estimate turns into equality for any polynomial
without zeros in the unit disk.
\end{corollary}

A well known theorem of Kronecker \cite{Kr} states that any monic
irreducible polynomial $P_n\in\Z_n[z],\ P_n(0)\neq 0,$ with all
zeros in the closed unit disk, must be cyclotomic. One can write
that statement in the form: $M(P_n)=1$ for such $P_n$ if and only if
$P_n$ is cyclotomic. A direct analog of this result exists for
$\|P_n\|_0.$

\begin{theorem} \label{thm1.3}
Suppose that $P_n\in\Z_n[z],\ P_n(0)\neq 0,$ is an irreducible
polynomial with all zeros in the closed unit disk. It is cyclotomic
if and only if $\|P_n\|_0=1.$
\end{theorem}

The next natural question is whether one can find a uniform lower
bound $\|P_n\|_0 \ge c > 1$ for all non-cyclotomic $P_n\in\Z_n[z],\
P_n(0)\neq 0.$ It is especially interesting in view of Lehmer's
conjecture, because of \eqref{1.1}. However, the answer to the
question is negative, as we show with the following example.

\begin{example} \label{ex1.4}
Consider $P_n(z)=nz^n-1.$ It has zeros $z_j,\ j=1,\ldots,n,$ that
are equally spaced on the circumference $|z|=n^{-1/n}.$ Note that
$M(P_n)=n$ and
\[
\|P_n\|_0 = n \exp\left(\frac{n(n^{-2/n}-1)}{2} \right),
\]
by \eqref{1.4}. Since
\[
n^{-2/n} = \exp\left(\frac{-2\log{n}}{n}\right) = 1 -
\frac{2\log{n}}{n} + O\left(\frac{\log^2{n}}{n^2}\right),
\]
we obtain that
\[
\|P_n\|_0 = \exp\left(O\left(\frac{\log^2{n}}{n}\right)\right) \to 1
\quad \mbox{as }n\to\infty.
\]

Similarly, we have for the reciprocal polynomial
$P_{2n}(z)=z^{2n}+nz^n+1$ that
\[
M(P_n) = \frac{n+\sqrt{n^2-4}}{2} \sim n \quad \mbox{as }n\to\infty,
\]
and
\[
\|P_n\|_0 = \frac{n+\sqrt{n^2-4}}{2} \exp\left(\frac{n}{2}\left(
\left( \frac{n-\sqrt{n^2-4}}{2} \right)^{2/n} - 1 \right)\right) \to
1 \quad \mbox{as }n\to\infty.
\]
\end{example}

One may notice that for both sequences of polynomials in this
example the zeros are asymptotically equidistributed near the unit
circumference. This is a part of a more general phenomenon discussed
in the next section.

We conclude this section with a remark on the arithmetic nature of
$\|P_n\|_0.$

\begin{proposition} \label{prop1.5}
If $P_n(z) = \sum_{k=0}^n a_k z^k \in\Z_n[z]$ has at least one zero
in $\D$, then $\|P_n\|_0$ is a transcendental number. Otherwise,
$\|P_n\|_0=M(P_n)=|a_0|$ is an integer.
\end{proposition}

\section{Asymptotic zero distribution}

Consider a polynomial $P_n(z)=a_n \prod_{j=1}^n (z-z_j)\in\C_n[z],$
and define its normalized zero counting measure by
\[
\nu_n:= \frac{1}{n} \sum_{j=1}^n \delta_{z_j},
\]
where $\delta_{z_j}$ is the unit pointmass at $z_j.$ Our main result
on the asymptotic zero distribution is as follows.

\begin{theorem} \label{thm2.1}
Suppose that $P_n\in\Z_n[z],\ \deg P_n = n,$ is a sequence of
polynomials without multiple zeros. If $\lim_{n\to\infty}
\|P_n\|_0^{1/n} = 1$ then $\nu_n$ converges to the normalized
arclength measure $d\theta/(2\pi)$ on $\T$ in the weak* topology, as
$n\to\infty.$
\end{theorem}

This result extends a theorem of Bilu \cite{Bi} for Mahler's
measure, see also Bombieri \cite{Bo} and Rumely \cite{Ru}. From a
more general point of view, Theorem \ref{thm2.1} is a descendant of
Jentzsch's result \cite{Je} on the asymptotic zero distribution of
the partial sums of a power series, and its generalization by
Szeg\H{o} \cite{Sz}. This area was further developed by Erd\H{o}s
and Tur\'an \cite{ET}, and by many others.

As an immediate application of Theorem \ref{thm2.1} we obtain a
result on the growth of $\|P_n\|_0$ for polynomials with restricted
zeros.

\begin{corollary} \label{cor2.2}
Suppose that $P_n\in\Z_n[z],\ \deg P_n = n,$ is a sequence of
polynomials with simple zeros contained in a closed set
$E\subset\C$. If $\T\not\subset E$ then there exists a constant
$C=C(E)>1$ such that
\[
\liminf_{n\to\infty} \|P_n\|_0^{1/n} \ge C > 1.
\]
\end{corollary}

This exhibits the geometric growth of $\|P_n\|_0$ for many families
of polynomials such as polynomials with real zeros, polynomials with
zeros in a sector, etc. Corresponding results with explicit bounds
for Mahler's measure were obtained by Schinzel \cite{Sc}, Langevin
\cite{La1, La2, La3}, Mignotte \cite{Mi}, Rhin and Smyth
\cite{RhSm}, Dubickas and Smyth \cite{DuSm}, and others.

In a somewhat different direction, we have the following result on
the asymptotic behavior of zeros.

\begin{theorem} \label{thm2.3}
Suppose that $P_n(z)=a_nz^n+\ldots+a_0\in\C_n[z],\ |a_0| \ge 1, \
n\in\N,$ is a sequence of polynomials.\\

(a) If $\lim_{n\to\infty} \|P_n\|_0 = 1$ then
\begin{align} \label{2.1}
\liminf_{n\to\infty} \min_{1 \le j \le n} |z_j| \ge 1.
\end{align}

(b) If $|a_n|\ge 1$ and $\lim_{n\to\infty} M(P_n)=1,$ then
\begin{align} \label{2.2}
\lim_{n\to\infty} \min_{1 \le j \le n} |z_j| = \lim_{n\to\infty}
\max_{1 \le j \le n} |z_j| = 1.
\end{align}
\end{theorem}

Thus part (a) of Theorem \ref{thm2.3} indicates that all zeros of
$P_n$ are pushed out of $\D$ as $n\to\infty,$ while in part (b) they
all tend to the unit circumference.

\section{Polynomial inequalities}

We discuss some general polynomial inequalities related to $M(P_n)$
and $\|P_n\|_0$ in this section. For a polynomial
$\Lambda_n(z)=\sum_{k=0}^n \lambda_k \binom{n}{k} z^k \in \C_n[z],$
consider the Szeg\H{o} composition with $P_n(z)=\sum_{k=0}^n a_k z^k
\in \C_n[z]:$
\begin{align} \label{3.1}
\Lambda P_n(z) := \sum_{k=0}^n \lambda_k a_k z^k.
\end{align}
If $\Lambda_n$ is a fixed polynomial, then $\Lambda P_n$ is a
multiplier operator acting on $P_n.$ More information on history and
applications of this composition may be found in \cite{dBS},
\cite{Ar1}, \cite{Ar2} and \cite{RS}. De Bruijn and Springer
\cite{dBS} proved a very interesting general inequality
\begin{align} \label{3.2}
M(\Lambda P_n) \le M(\Lambda_n) M(P_n),
\end{align}
which did not receive the attention it truly deserves. In
particular, it contains the inequality
\[
M(P_n') \le n M(P_n)
\]
that is usually attributed to Mahler, who proved it later in
\cite{Ma}. To see this, just note that if
$\Lambda_n(z)=nz(1+z)^{n-1} = \sum_{k=0}^n k \binom{n}{k} z^k,$ then
$\Lambda P_n(z)=zP_n'(z)$ and $M(\Lambda_n)=n.$ Furthermore,
\eqref{3.2} immediately answers a question about a lower bound for
Mahler's measure of derivative raised in \cite[pp. 12 and 194]{EW},
see \cite{St}. For $P_n'(z)=\sum_{k=0}^{n-1} a_k z^k,$ write
\begin{align*}
\frac{1}{z}\left(P_n(z)-P_n(0)\right) = \sum_{k=0}^{n-1}
\frac{a_k}{k+1} z^k = \Lambda P_n'(z),
\end{align*}
where
\[
\Lambda_{n-1}(z) = \sum_{k=0}^{n-1} \frac{1}{k+1} \binom{n-1}{k} z^k
= \frac{(1+z)^n-1}{nz}.
\]
The result of de Bruijn and Springer gives
\[
M(P_n(z)-P_n(0)) \le M(\Lambda_{n-1})\, M(P_n'),
\]
with
\[
M(\Lambda_{n-1}) = \frac{1}{n} M\left((1+z)^n-1\right) = \frac{1}{n}
\prod_{n/6<k<5n/6} 2\sin\frac{k\pi}{n}.
\]
There are many other interesting consequences of \eqref{3.2}, which
we leave for the reader.

We obtain the following generalization of \eqref{3.2} for
$\|P_n\|_0.$

\begin{theorem} \label{thm3.1}
For any $\Lambda_n\in\C_n[z]$ and any $P_n\in\C_n[z],$ we have
\begin{align} \label{3.3}
\|\Lambda P_n\|_0 \le M(\Lambda_n) \|P_n\|_0.
\end{align}
\end{theorem}
Note that equality holds in \eqref{3.2} and \eqref{3.3} for any
polynomial $P_n\in\C_n[z]$ when $\Lambda_n(z)=(1+z)^n=\sum_{k=0}^n
\binom{n}{k} z^k,$ because $\Lambda P_n\equiv P_n$ and
$M((1+z)^n)=1.$ This inequality allows to treat many problems in a
unified way, and it has interesting corollaries stated below.

First, we mention an analog of the de Bruijn-Springer-Mahler
inequality.
\begin{corollary} \label{cor3.2}
For any $P_n\in\C_n[z]$, we have that
\begin{align} \label{3.4}
\|zP_n'\|_0 \le n \|P_n\|_0
\end{align}
and
\begin{align} \label{3.5}
\|P_n'\|_0 \le \sqrt{e}\, n\, \|P_n\|_0,
\end{align}
where equality holds for $P_n(z)=z^n.$
\end{corollary}

Another consequence relates $\|P_n\|_0$ to the coefficients of
$P_n.$
\begin{corollary} \label{cor3.3}
If $P_n(z)=\sum_{k=0}^n a_k z^k \in \C_n[z]$ then
\begin{align} \label{3.6}
|a_k| \le  e^{k/2}\binom{n}{k}\, \|P_n\|_0,\quad k=0,\ldots,n.
\end{align}
\end{corollary}

Recall that we have $|a_k| \le \binom{n}{k} M(P_n)$ for Mahler's
measure (see, e.g., \cite{EW}), which follows from \eqref{3.2} by
letting $\Lambda_n(z)=\binom{n}{k} z^k$. One can certainly continue
with a list of corollaries by choosing proper polynomials
$\Lambda_n.$

\section{Approximation by polynomials with integer coefficients}

We consider a related question of approximation by polynomials with
integer coefficients on the unit disk. There is a well known
condition necessary for approximation by integer polynomials in
essentially any norm on $\D$.

\begin{proposition} \label{prop4.1}
Suppose that $P_n\in\Z_n[z],\ n\in\N,$ converge to $f$ uniformly on
compact subsets of $\D$. Then  $f$ is analytic in $\D$ and
$f^{(k)}(0)/k! \in \Z \
 \forall\, k\ge 0,\ k\in\Z.$
\end{proposition}

This necessary condition for the convergence is clearly equivalent
to the fact that the power series expansion of $f$ at the origin has
integer coefficients.

Define the Hardy space norm on $\D$ by
\[
\|P_n\|_{H^p} := \left(\frac{1}{2\pi} \int_0^{2\pi}
|P_n(e^{i\theta})|^p\,d\theta \right)^{1/p}, \quad 0<p<\infty.
\]
It is well known that approximation by polynomials with integer
coefficients is possible in $H^p$ only in the trivial case, see
\cite{Fer} and \cite{Tri}. More precisely, we have

\begin{proposition} \label{prop4.2}
Suppose that $f\in H^p,\ 0 < p \le \infty.$ If $P_n\in\Z_n[z],\
n\in\N$, satisfy
\begin{align} \label{4.1}
\lim_{n\to\infty} \|f-P_n\|_{H^p} = 0,
\end{align}
then $f$ is a polynomial with integer coefficients.
\end{proposition}
It appears an open question whether this proposition is true for
$p=0$, i.e. for approximation of functions in Mahler's measure.
Generally, nontrivial approximation by integer polynomials in the
supremum norm is valid on sets with transfinite diameter (capacity)
less than 1 \cite{Fer,Tri, FerN}, and it is not possible if the
transfinite diameter is greater than or equal to 1. But the
transfinite diameter of $\D$ is exactly equal to 1, so that we deal
with a borderline case. However, we show that the Bergman space
$A^p$ is different from the Hardy space $H^p$ in this regard, as it
does allow approximation by polynomials with integer coefficients.

\begin{theorem} \label{thm4.3}
Suppose that $f\in A^p,\ 1<p<\infty.$ We have
\begin{align} \label{4.2}
\lim_{n\to\infty} \|f-P_n\|_p = 0,
\end{align}
for a sequence of polynomials $P_n\in\Z_n[z],\ n\in\N$, if and only
if $f$ has a power series expansion about $z=0$ with integer
coefficients. Clearly, this is equivalent to $f^{(k)}(0)/k! \in \Z \
 \forall\, k\ge 0,\ k\in\Z.$
\end{theorem}
Thus there are many functions in $A^p$ that can be approximated by
polynomials with integer coefficients. In fact, one can use partial
sums of the power series for this purpose, see the proof of Theorem
\ref{thm4.3}. However, we do not know whether Theorem \ref{thm4.3}
is valid in the case $0\le p\le 1$. Note that if $f\in A^p,\ p>1,$
has a Taylor expansion with integer coefficients, then $f\in A^q$
for any $q\in[0,p)$ and the partial sums $P_n$ of this expansion
satisfy $\|f-P_n\|_q \le \|f-P_n\|_p \to 0$ as $n\to\infty.$

\section{Multivariate polynomials}

The definition of $\|P_n\|_0$ is easily generalized to the case of
multivariate polynomials $P_n(z_1,\ldots,z_d)$ as follows:
\[
\|P_n\|_0 := \exp\left(\frac{1}{\pi^d} \int_{\D}\ldots\int_{\D}
\log|P_n(z_1,\ldots,z_d)| \,dA(z_1)\ldots dA(z_d)\right).
\]
It is also parallel to multivariate Mahler's measure
\[
M(P_n) := \exp\left(\frac{1}{(2\pi)^d} \int_{\T}\ldots\int_{\T}
\log|P_n(z_1,\ldots,z_d)| \,|d z_1|\ldots |d z_d|\right).
\]
We note that many of the properties of $\|P_n\|_0$ are preserved in
the multivariate case. Thus it still defines a multiplicative height
on the space of polynomials. If $P_n$ is a polynomial with complex
coefficients and the constant term $a_0,$ then we can apply the area
mean value inequality to the (pluri)subharmonic function
$\log|P_n(z_1,\ldots,z_d)|$ in each variable, which gives together
with Fubini's theorem that
\[
\|P_n\|_0 \ge |a_0|.
\]
Furthermore, the above inequality turns into equality if
$P_n(z_1,\ldots,z_d)\neq 0$ on $\D^d,$ by the area mean value
theorem for the (pluri)harmonic function
$\log|P_n(z_1,\ldots,z_d)|$. However, it is rather unlikely that
some kind of explicit relation such as \eqref{1.4} exists for
general multivariate polynomials.

We now state an estimate generalizing Corollary \ref{cor1.2}.
\begin{proposition} \label{prop5.1}
For a polynomial
\begin{align} \label{5.1}
P_n(z_1,\ldots,z_d) = \sum_{k_1+\ldots+k_d\le n} a_{k_1\ldots k_d}
z_1^{k_1} \ldots z_d^{k_d}
\end{align}
of degree at most $n$ with complex coefficients, we have
\begin{align} \label{5.2}
e^{-n/2} M(P_n) \le \|P_n\|_0 \le M(P_n).
\end{align}
Equality holds in the lower estimate for any $P_n(z_1,\ldots,z_d) =
a_{k_1\ldots k_d} z_1^{k_1} \ldots z_d^{k_d}$ with $k_1+\ldots+k_d =
n.$ The upper estimate turns into equality for any polynomial not
vanishing in $\D^d.$
\end{proposition}

It is of interest to find explicit values of the multivariate
$\|P_n\|_0.$ This problem has received a considerable attention in
Mahler's measure setting (see \cite{Boy81}, \cite{Sm1,Sm2},
\cite{EW}, \cite{GH}), and it remains a very active area of
research. In particular, it is of importance to characterize
multivariate polynomials with integer coefficients satisfying
$\|P_n\|_0=1.$ Smyth \cite{Sm2} proved a complete Kronecker-type
characterization for the multivariate Mahler's measure $M(P_n)=1$.
Thus we expect that one should be able to produce an analog for
$\|P_n\|_0,$ generalizing Theorem \ref{thm1.3}. We postpone a
detailed study of the multivariate $\|P_n\|_0$ for another occasion,
and conclude with simple examples.

\begin{example} \label{ex5.2}

The following identities hold for the multivariate $\|P_n\|_0$:

{\noindent\rm (a)} $\|z_1+z_2\|_0 = e^{-1/4}$

{\noindent\rm (b)} $\|1+z_1^{k_1} \ldots z_d^{k_d}\|_0 = 1, \quad
k_1, \ldots, k_d \ge 0$

{\noindent\rm (c)} If the polynomial $P_n$ of the form \eqref{5.1}
satisfies
\[
|a_{0\ldots 0}| \ge \sum_{0 < k_1+\ldots+k_d \le n} |a_{k_1\ldots
k_d}|,
\]
then $\|P_n\|_0 = M(P_n) = |a_{0\ldots 0}|.$

\end{example}

\section{Proofs}

\subsection{Proofs for Section 1}

\begin{proof}[Proof of Theorem \ref{thm1.1}]

If $P_n$ does not vanish in $\D,$ then $\log|P_n(z)|$ is harmonic in
$\D.$ Hence $M(P_n)=|a_0|$ and $\|P_n\|_0=|a_0|$ follow from the
contour and area mean value theorems. Assume now that $P_n$ has
zeros in $\D.$ Applying Jensen's formula, we obtain that
\[
\log M(P_n) = \frac{1}{2\pi}\int_0^{2\pi}
\log|P_n(e^{i\theta})|\,d\theta = \log|a_n| + \sum_{|z_j|\ge 1}
\log|z_j|.
\]
Furthermore,
\begin{align*}
\log \|P_n\|_0 &= \frac{1}{\pi} \int_0^1 \int_0^{2\pi}
\log|P_n(re^{i\theta})|\,rdrd\theta \\ &= 2 \int_0^1 \left(
\frac{1}{2\pi} \int_0^{2\pi} \log|P_n(re^{i\theta})|\,d\theta\right)rdr \\
&= 2 \int_0^1 \left(\log|a_n| + \sum_{|z_j|\ge r} \log|z_j| +
\sum_{|z_j| < r} \log r \right)rdr \\
&= \log|a_n| + \sum_{|z_j|\ge 1} \log|z_j| + \frac{1}{2} \sum_{|z_j|
< 1} (|z_j|^2 - 1).
\end{align*}
Hence
\[
\|P_n\|_0 = M(P_n)\, \exp\left( \frac{1}{2} \sum_{|z_j| < 1}
(|z_j|^2-1) \right).
\]

\end{proof}

\begin{proof}[Proof of Corollary \ref{cor1.2}]

Inequality \eqref{1.5} follows from \eqref{1.4} after observing that
the smallest value of the exponential is achieved when all $z_j=0$,
while the largest value is 1 when all $|z_j|\ge 1.$

\end{proof}

\begin{proof}[Proof of Theorem \ref{thm1.3}]

If $P_n$ is cyclotomic, then $\|P_n\|_0 = 1$ by Theorem
\ref{thm1.1}, because $|z_j|=1, \ j=1,\ldots,n,$ and $M(P_n)=1$.
Assume now that $\|P_n\|_0 = 1$. Let $z_j, \ j=1,\ldots,m,\ m\le n,$
be the zeros of $P_n$ in $\D.$ We recall that $M(P_n) = |a_n|
\prod_{|z_j|>1} |z_j| = |a_0| \prod_{|z_j|<1} |z_j|^{-1},$ where
$a_0\neq 0$ is the constant term of $P_n$. Thus we have from
\eqref{1.4} that
\begin{align} \label{6.1}
\|P_n\|_0 = |a_0|\, \prod_{j=1}^m \frac{e^{(|z_j|^2-1)/2}}{|z_j|}
\ge \prod_{j=1}^m \frac{e^{(|z_j|^2-1)/2}}{|z_j|}.
\end{align}
Define $g(x):=e^{(x^2-1)/2}/x,\ x>0,$ and observe that $g'(x)<0$
when $x\in(0,1),$ while $g'(x)>0$ when $x\in(1,\infty).$ Hence
\begin{align} \label{6.2}
g(1)=1 \mbox{ is the strict global minimum for $g(x)$ on
$(0,\infty)$.}
\end{align}
It follows from \eqref{6.1}-\eqref{6.2} that
\[
1 < \prod_{j=1}^m g(|z_j|) = \prod_{j=1}^m
\frac{e^{(|z_j|^2-1)/2}}{|z_j|} \le \|P_n\|_0 = 1,
\]
which is a contradiction. Hence $P_n$ has no zeros in $\D$, and
$M(P_n)=\|P_n\|_0=1$ by Theorem \ref{thm1.1}. This implies that
$P_n$ is cyclotomic by Kronecker's theorem.

We could also proceed in a different way, by assuming that
$\|P_n\|_0 = 1$ and observing from \eqref{6.1} that
\[
\exp\left(\sum_{j=1}^m \frac{|z_j|^2-1}{2}\right) = \frac{1}{|a_0|}
\prod_{j=1}^m |z_j|
\]
Since the expression on the right is an algebraic number, as well as
the sum in the exponent on the left, we obtain that equality is only
possible when the latter sum is zero, by the well known result of
Lindemann that the exponential of a nonzero algebraic number is
transcendental \cite{Ba}. Hence $|z_j|\ge 1, \ j=1,\ldots,n,$ and
$M(P_n)=\|P_n\|_0=1$ as before.

\end{proof}

\begin{proof}[Proof of Proposition \ref{prop1.5}]

Assume that the zeros of $P_n$ in $\D$ are given by $z_j,\
j=1,\ldots,m,$ and observe from \eqref{1.4} that
\[
\exp\left(\sum_{j=1}^m \frac{|z_j|^2-1}{2}\right) =
\frac{\|P_n\|_0}{M(P_n)}
\]
Since the sum in the exponent on the left is algebraic, we obtain
that the left hand side is transcendental by the well known result
of Lindemann \cite{Ba}. Note that $M(P_n)$ is always algebraic. If
$\|P_n\|_0$ were algebraic, then the right hand side would be
algebraic too, a contradiction. When $P_n$ has no zeros in $\D$, we
have $\|P_n\|_0=M(P_n)=|a_0|$ by Theorem \ref{thm1.1}.

\end{proof}

\subsection{Proofs for Section 2}

\begin{proof}[Proof of Theorem \ref{thm2.1}]

We first show that $P_n$ has $o(n)$ zeros in $D_r:=\{z:|z|<r\}$ as
$n\to\infty$, for any $r<1.$ Assume to the contrary that there is a
subsequence of $n$ such that $P_n$ has at least $\alpha n$ zeros,
with $\alpha>0,$ in some $D_r,\ r<1.$ Suppose that those zeros are
$z_j\neq 0, \ j=1,\ldots,m,\ m\le n,$ and proceed as in the proof of
Theorem \ref{thm1.3} to obtain
\begin{align} \label{6.3}
\prod_{j=1}^m g(|z_j|) = \prod_{j=1}^m
\frac{e^{(|z_j|^2-1)/2}}{|z_j|} \le \|P_n\|_0
\end{align}
by \eqref{6.1}. Since $g(x)=e^{(x^2-1)/2}/x$ is strictly decreasing
on $(0,1)$, we have that
\[
\prod_{j=1}^m g(|z_j|) \ge (g(r))^{\alpha n}.
\]
It immediately follows from \eqref{6.2} and \eqref{6.3} that
\[
\limsup_{n\to\infty} \|P_n\|_0^{1/n} \ge (g(r))^{\alpha} > 1,
\]
which is in direct conflict with assumptions of this theorem. If
$P_n$ has a simple zero at $z=0,$ then $P_n(z)=z Q_{n-1}(z)$ and
$\|P_n\|_0 = \|Q_{n-1}\|_0/\sqrt{e}.$ Hence we can apply the above
argument to $Q_{n-1}$ and come to the same conclusion that $P_n$ has
$o(n)$ zeros in $D_r:=\{z:|z|<r\},\ r<1,$ as $n\to\infty$.

The second step is to show that $\lim_{n\to\infty} (M(P_n))^{1/n} =
1.$ Note that
\begin{align} \label{6.4}
1 \le M(P_n) = \|P_n\|_0 \, \exp\left(\frac{1}{2} \sum_{|z_j| < 1}
(1-|z_j|^2) \right).
\end{align}
If $P_n$ has $m=o(n)$ zeros in $D_r,\ r<1,$ then
\[
\exp\left(\frac{1}{2} \sum_{|z_j| < 1} (1-|z_j|^2) \right) \le
e^{m/2+n(1-r^2)/2}.
\]
Using this in \eqref{6.4}, we obtain that
\begin{align*}
1 &\le \liminf_{n\to\infty} (M(P_n))^{1/n} \le \limsup_{n\to\infty}
(M(P_n))^{1/n} \\ &\le e^{(1-r^2)/2} \lim_{n\to\infty}
\|P_n\|_0^{1/n} = e^{(1-r^2)/2}.
\end{align*}
Hence $\lim_{n\to\infty} (M(P_n))^{1/n} = 1$ follows by letting
$r\to 1-.$ The proof may now be completed by applying Bilu's result
\cite{Bi} (at least when $P_n$ is irreducible for all $n\in\N$), but
we prefer to continue with an independent proof via a standard
potential theoretic argument.

Observe that $P_n(z)=a_n \prod_{j=1}^n (z-z_j)$ has $o(n)$ zeros in
$\C\setminus D_r,\ r>1,$ for otherwise we would have
$\liminf_{n\to\infty} (M(P_n))^{1/n} > 1$ as
\[
M(P_n) = |a_n| \prod_{|z_j|>1} |z_j| \ge \prod_{|z_j|>1} |z_j|.
\]
This also implies that
\begin{align} \label{6.5}
\lim_{n\to\infty} |a_n|^{1/n} = 1.
\end{align}
Hence any weak* limit $\nu$ of the sequence $\nu_n$ must satisfy
$\supp\, \nu \subset \T.$ Define the logarithmic energy of $\nu$ by
\[
I(\nu):=\iint \log \frac{1}{|z-t|} d\nu(z)\,d\nu(t).
\]
Our goal is to show that $I(\nu) = 0,$ which implies that $\nu$ has
the smallest possible energy among all positive Borel measures of
mass 1 supported on $\T.$ On the other hand, it is well known in
potential theory that the equilibrium measure minimizing the energy
integral is unique, and it is equal to the normalized arclength on
$\T$ \cite{Ra,Ts}. Thus $\nu=d\theta/(2\pi)$ and the proof would be
completed.

Define the discriminant of $P_n$ as $\Delta_n:= a_n^{2n-2}
\prod_{1\le j<k\le n} (z_j-z_k)^2$. Observe that it is an integer,
being a symmetric form with integer coefficients in the roots of
$P_n\in\Z_n[z]$. Since $P_n$ has no multiple roots, we have
$\Delta_n\neq 0$ and $|\Delta_n|\ge 1.$ Therefore,
\begin{align} \label{6.6}
\log\frac{1}{|\Delta_n|} = -(2n-2)\log|a_n| + \sum_{j\neq k}
\log\frac{1}{|z_j-z_k|} \le 0.
\end{align}
Let
\[
K_M(z,t):=\min\left(\log\frac{1}{|z-t|},M\right), \qquad M>0.
\]
It is clear that $K_M(z,t)$ is a continuous function in $z$ and $t$
on $\C\times\C$, and that $K_M(z,t)$ increases to
$\log\frac{1}{|z-t|}$ as $M\to\infty.$ Using the Monotone
Convergence Theorem and the weak* convergence of $\nu_n\times\nu_n$
to $\nu\times\nu,$ we obtain that
\begin{align*}
I(\nu) &= \lim_{M\to\infty} \iint K_M(z,t)\, d\nu(z)\,d\nu(t) \\ &=
\lim_{M\to\infty} \left( \lim_{n\to\infty} \iint K_M(z,t)\,
d\nu_n(z)\,d\nu_n(t) \right) \\ &= \lim_{M\to\infty} \left(
\lim_{n\to\infty} \left( \frac{1}{n^2} \sum_{j\neq k} K_M(z_j,z_k)
+\frac{M}{n} \right) \right) \\ &\le \lim_{M\to\infty} \left(
\liminf_{n\to\infty} \frac{1}{n^2} \sum_{j\neq k}
\log\frac{1}{|z_j-z_k|} \right) \\ &= \liminf_{n\to\infty}
\frac{1}{n^2} \log\frac{|a_n|^{2n-2}}{\Delta_n}.
\end{align*}
Hence $I(\nu)\le 0$ follows from \eqref{6.5}-\eqref{6.6}. But
$I(\mu)>0$ for any positive unit Borel measure supported on $\T,$
with the only exception for the equilibrium measure $d\mu_{\T}:=
d\theta/(2\pi)$, $I(\mu_{\T})=0,$ see \cite[pp. 53-89]{Ts}.

\end{proof}

\begin{proof}[Proof of Theorem \ref{thm2.3}]
$(a)$ We use the same notation and approach as in the proof of
Theorem \ref{thm1.3}. If $P_n$ has no zeros in $\D,$ then $\min_{1
\le j \le n} |z_j| \ge 1.$ Otherwise, let $z_j, \ j=1,\ldots,m,\
m\le n,$ be the zeros of $P_n$ in $\D.$ It follows from
\eqref{6.1}-\eqref{6.2} that
\[
\|P_n\|_0 = |a_0|\, \prod_{j=1}^m \frac{e^{(|z_j|^2-1)/2}}{|z_j|}
\ge g\left(\min_{1 \le j \le n} |z_j|\right)>1.
\]
Thus we obtain the result by the continuity of
$g(x)=e^{(x^2-1)/2}/x,\ x>0,$ and \eqref{6.2}.

$(b)$ Note that $\lim_{n\to\infty} \|P_n\|_0 = 1$ in this case too,
by \eqref{1.1} and \eqref{1.2}. Hence \eqref{2.1} holds true.
Furthermore, we have for any zero $z_k\in\C\setminus\D$ that
\[
1 \le |z_k| \le |a_n| \prod_{|z_j|>1} |z_j| = M(P_n).
\]
Thus
\[
\lim_{n\to\infty} \max_{1 \le j \le n} |z_j| = 1,
\]
and \eqref{2.2} follows.

\end{proof}

\subsection{Proofs for Section 3}

\begin{proof}[Proof of Theorem \ref{thm3.1}]

Using \eqref{3.2} for the polynomial $P_n(rz),\ r\in[0,1]$, we
obtain that
\[
\int_0^{2\pi} \log|\Lambda P_n(re^{i\theta})|\, d\theta \le 2\pi
\log M(\Lambda_n) + \int_0^{2\pi} \log|P_n(re^{i\theta})|\, d\theta.
\]
Hence \eqref{3.3} follows immediately, if we multiply this
inequality by $r\,dr/\pi$ and integrate from 0 to 1.

\end{proof}

\begin{proof}[Proof of Corollary \ref{cor3.2}]

We follow \cite{dBS} by setting $\Lambda_n(z)=nz(1+z)^{n-1} =
\sum_{k=0}^n k \binom{n}{k} z^k.$ This gives $\Lambda
P_n(z)=zP_n'(z)$ and $M(\Lambda_n)=n.$ Hence \eqref{3.4} is a
consequence of \eqref{3.3}. In order to deduce \eqref{3.5} from
\eqref{3.4}, we only need to observe that $\|zP_n'\|_0 = \|z\|_0
\|P_n'\|_0 = \|P_n'\|_0 / \sqrt{e}.$

\end{proof}

\begin{proof}[Proof of Corollary \ref{cor3.3}]

Let $\Lambda_n(z)=\binom{n}{k} z^k, 0 \le k \le n,$ where $k$ is
fixed. Then $\Lambda P_n(z)=a_k z^k$ and
$M(\Lambda_n)=\binom{n}{k}.$ It follows from \eqref{3.3} that
\[
\|a_k z^k\|_0 = |a_k| e^{-k/2} \le \binom{n}{k} \|P_n\|_0,
\]
because $\|z^k\|_0 = e^{-k/2}.$

\end{proof}

\subsection{Proofs for Section 4}

\begin{proof}[Proof of Proposition \ref{prop4.1}]

Recall that the uniform convergence of $P_n$ to $f$ on compact
subsets of $\D$ implies that $f$ is analytic in $\D,$ and that
$P_n^{(k)}$ converge to $f^{(k)}$ on compact subsets of $\D$ for any
$k\in\N.$ In particular,
\[
\lim_{n\to\infty} P_n^{(k)}(0) = f^{(k)}(0) \qquad \forall\ k\ge 0,\
k\in\Z.
\]
But $P_n^{(k)}(0)=k!a_k,$ where $a_k\in\Z$ is a corresponding
coefficient of $P_n.$ Hence the result follows.

\end{proof}

\begin{proof}[Proof of Proposition \ref{prop4.2}]

We have that
\[
\|P_n-P_{n-1}\|_{H^p} \le \|f-P_n\|_{H^p} + \|f-P_{n-1}\|_{H^p}
\]
by the triangle inequality for $p\ge 1$, and
\[
\|P_n-P_{n-1}\|_{H^p}^p \le \|f-P_n\|_{H^p}^p +
\|f-P_{n-1}\|_{H^p}^p
\]
for $0<p<1.$ In both cases, \eqref{4.1} implies that
\[
\lim_{n\to\infty} \|P_n-P_{n-1}\|_{H^p} = 0, \qquad 0<p\le\infty.
\]
If $P_n\not\equiv P_{n-1}$ then we let $a_k z^k$ be the lowest
nonzero term of $P_n-P_{n-1}$, where $|a_k|\in\N.$ Using the mean
value inequality \cite{DS}, we obtain
\[
\|P_n-P_{n-1}\|_{H^p} \ge |a_k| \ge 1, \qquad 0<p\le\infty.
\]
This is obviously impossible as $n\to\infty,$ so that we have
$P_n\equiv P_{n-1}$ for all sufficiently large $n\in\N$. Hence the
limit function $f$ is also a polynomial with integer coefficients.

\end{proof}

\begin{proof}[Proof of Theorem \ref{thm4.3}]

If \eqref{4.2} holds then $P_n$ converge to $f$ on compact subsets
of $\D$ by the area mean value inequality:
\begin{align*}
|f(z) - P_n(z)|^p &\le \frac{1}{\pi (1-|z|)^2} \iint_{|t-z|<1-|z|}
|f(t) - P_n(t)|^p\, dA \\ &\le \frac{\|f(t) - P_n(t)\|_p^p}
{(1-|z|)^2} \to 0, \quad n\to\infty,\ z\in\D.
\end{align*}
Hence $f$ has a power series expansion at $z=0$ with integer
coefficients by Proposition \ref{prop4.1}.

Conversely, suppose that $f\in A^p$ is represented by a power series
with integer coefficients. Since the partial sums of this series
converge to $f$ in $A^p$ norm for $1<p<\infty$ by Theorem 4 \cite[p.
31]{DS}, we can select the sequence $P_n$ be the sequence of the
partial sums.

\end{proof}

\subsection{Proofs for Section 5}

\begin{proof}[Proof of Proposition \ref{prop5.1}]

We apply \eqref{1.5} in each variable $z_j,\ j=1,\ldots,d,$ and use
Fubini's theorem to prove \eqref{5.2}. Indeed, \eqref{1.5} gives
that
\begin{align*}
\frac{1}{2\pi} \int_{\T} \log|P_n(z_1,\ldots,z_d)| \,|d z_1| -
\frac{k_1}{2} &\le \frac{1}{\pi} \int_{\D} \log|P_n(z_1,\ldots,z_d)|
\,dA(z_1) \\ &\le \frac{1}{2\pi} \int_{\T} \log|P_n(z_1,\ldots,z_d)|
\,|d z_1|
\end{align*}
is true for all $z_2,\ldots,z_d\in\C.$ Integrating the above
inequality with respect to $dA(z_2)/\pi$, interchanging the order of
integration in the lower and upper bounds, and applying \eqref{1.5}
in the variable $z_2$, we obtain
\begin{align*}
&\frac{1}{(2\pi)^2} \int_{\T} \int_{\T} \log|P_n(z_1,\ldots,z_d)|
\,|dz_1| |dz_2|- \frac{k_1+k_2}{2} \\ &\le \frac{1}{\pi^2} \int_{\D}
\int_{\D} \log|P_n(z_1,\ldots,z_d)| \,dA(z_1) dA(z_2) \\ &\le
\frac{1}{(2\pi)^2} \int_{\T} \int_{\T} \log|P_n(z_1,\ldots,z_d)|
\,|dz_1| |dz_2|
\end{align*}
is true for all $z_3,\ldots,z_d\in\C.$ After carrying out this
argument for each variable $z_j$, we arrive at \eqref{5.2} in $d$
steps. When $P_n(z_1,\ldots,z_d)\neq 0$ in $\D^d$, we have that
$\|P_n\|_0 = M(P_n) = |a_{0\ldots 0}|$ by the iterative application
of Theorem \ref{thm1.1}. If $P_n(z_1,\ldots,z_d) = a_{k_1\ldots k_d}
z_1^{k_1} \ldots z_d^{k_d}$, where $k_1+\ldots+k_d = n,$ then we
evaluate directly that $M(P_n) = |a_{k_1\ldots k_d}|$ and $\|P_n\|_0
= |a_{k_1\ldots k_d}| e^{-n/2}$, because $\|z_j\|_0 = e^{-1/2},\
j=1,\ldots,n.$

\end{proof}

\begin{proof}[Proof of Example \ref{ex5.2}]

(a) Applying \eqref{1.4}, we have that
\begin{align*}
\frac{1}{\pi^2} \int_{\D} \int_{\D} \log|z_1+z_2| \,dA(z_1) dA(z_2)
= \frac{1}{\pi} \int_{\D} \frac{|z_2|^2-1}{2} \,dA(z_2) =
-\frac{1}{4}.
\end{align*}
(b) is an immediate consequence of (c).\\
(c) Let $a_{0\ldots 0} = |a_{0\ldots 0}| e^{i\phi}.$ Observe that
$P_n(z_1,\ldots,z_d) + \varepsilon e^{i\phi} \neq 0$ in $\D^d$ for
any $\varepsilon>0$, because
\[
|P_n(z_1,\ldots,z_d) + \varepsilon e^{i\phi}| \ge |a_{0\ldots 0}| +
\varepsilon - \sum_{0 < k_1+\ldots+k_d \le n} |a_{k_1\ldots k_d}| >
0
\]
by the triangle inequality. We obtain that $\|P_n + \varepsilon
e^{i\phi}\|_0 = M(P_n + \varepsilon e^{i\phi}) = |a_{0\ldots 0}| +
\varepsilon$ by the area and contour mean value properties of the
(pluri)harmonic function $\log|P_n(z_1,\ldots,z_d) + \varepsilon
e^{i\phi}|$ in $\D^d$, and the result follows by letting
$\varepsilon \to 0.$

\end{proof}

\end{document}